\documentclass[reqno,12pt,a4paper]{amsart}
\usepackage{amscd,amsfonts,amssymb,amsthm,latexsym}
\usepackage{mathtools}
\usepackage{bm}
\usepackage{srcltx}

\theoremstyle{plain}
\newtheorem{theorem}{Theorem}[section]
\newtheorem*{theorem*}{Theorem}
\newtheorem{corollary}{Corollary}[section]
\newtheorem{lemma}{Lemma}[section]

\newtheorem{hypothesis}{Hypothesis}

\newtheorem*{conjecture*}{Conjecture}

\theoremstyle{definition}

\newtheorem*{definition*}{Definition}

\theoremstyle{remark}

\newtheorem*{remark*}{Remark}

\numberwithin{equation}{section}

\begin{document}
\raggedbottom 

\title[Primes in tuples and Romanoff's theorem]{Primes in tuples and Romanoff's theorem}

\author{Artyom Radomskii}

\begin{abstract}We obtain a lower bound for a number of primes in tuples. As applications, we obtain a lower bound for the Romanoff type representation functions.
\end{abstract}

 \address{HSE University, Moscow, Russia}


\email{artyom.radomskii@mail.ru}

\maketitle

\section{Introduction}

Let $\mathcal{L}=\{L_{1},\ldots, L_{k}\}$ be a set of distinct linear functions $L_{i}(n)=a_{i}n+b_{i}$, $i=1,\ldots, k,$ with coefficients in the positive integers. We recall that such a set is called \emph{admissible} if for every prime $p$ there is an integer $n_{p}$ such that $(\prod_{i=1}^{k}L_{i}(n_{p}), p)=1$. The famous \emph{Prime $k$-tuples conjecture} asserts that if $\mathcal{L}=\{L_{1},\ldots, L_{k}\}$ is admissible, then there are infinitely many integers $n$ such that all $L_{i}(n)$ ($1\leq i \leq k$) are prime.

Although such a conjecture remains far from resolution, recent progress \cite{Zhang}, \cite{Maynard1}, \cite{Maynard} has enabled to prove weak forms of this conjecture. It follows from results of Zhang \cite{Zhang} and Maynard \cite{Maynard1}, \cite{Maynard} that there are infinitely many integers $n$ such that \emph{several} (rather than \emph{all}) of the $L_{i}(n)$ are primes. Also, Maynard \cite{Maynard} obtained quantitative results in the case when the coefficients $a_i$, $b_i$ and the number $k$ of functions in $\mathcal{L}$ vary slightly with $x$.

Based on the Maynard technique, in \cite{Chen.Ding} Chen and Ding proved the following result.

\begin{theorem}[\mbox{\cite[Theorem 1.5]{Chen.Ding}}]\label{T.Ch.Ding}
Let $0<\alpha< 1/4$ and $\beta > 0$. Then there exist a positive constant $C_{0}$ depending only on $\alpha$ such that for sufficiently large number $x$ and $C_{0} \leq k \leq (\ln x)^{\alpha}$, if $\{a_{1}n+b_{1},\ldots, a_{k}n+b_{k}\}$ is a set of $k$ distinct linear functions with $(a_{i}, b_{i})=1,$ $0< a_{i}\leq (\ln x)^{\beta},$ $P^{+}(a_{i})\leq 0.9 \ln\ln x$ and $0< b_{i}\leq x^{\alpha}$ for all $1\leq i \leq k$, then
\[
\#\big\{x\leq n < 2x: \#(\{a_{1}n+b_{1},\ldots, a_{k}n+b_{k}\}\cap \mathbb{P}) \geq C_{0}^{-1} \ln k\big\}\gg
x\,\textup{exp}(-C_{0} (\ln x)^{\alpha}).
\]Here $P^{+}(n)$ denotes the largest prime factor of $n$.
\end{theorem}

 The following theorem is the main result of our paper.

\begin{theorem}\label{T1}
There are positive absolute constants $C$, $c$, $c_{0}$, and $c_{1}$ such that the following holds. Let $x\in \mathbb{R}$ and $k\in \mathbb{N}$ be such that $x\geq 10$, $k \geq C$ and $k \leq c(\ln x)^{1/4}(\ln\ln x)^{-1/2}$. Then there is a prime $p_{0}$ such that
\[
0.9\ln\ln x< p_{0}\leq \textup{exp} (c_1 \sqrt{\ln x})
\]and the following holds. Let $\mathcal{L}=\{L_{1},\ldots, L_{k}\}$ be a set of distinct linear functions $L_{i}(n)=a_{i}n+b_{i},$ $i=1,\ldots, k,$ with coefficients in the positive integers. Let the coefficients $L_{i}(n)=a_{i}n+b_{i}\in \mathcal{L}$ satisfy $(a_{i}, b_{i})=1,$
 \[
  a_{i}\leq \textup{exp}(c_{0}\sqrt{\ln x}),\qquad (a_{i}, p_0)=1,\qquad b_{i}\leq x\,\textup{exp}(c_{0}\sqrt{\ln x})
 \]for all $1\leq i \leq k$. Then
 \[
 \#\{x\leq n< 2x: \#(\{L_{1}(n),\ldots, L_{k}(n)\}\cap \mathbb{P}) \geq C^{-1}\ln k\}\geq \frac{x}{(\ln x)^{k}\textup{exp}(Ck)}.
 \]
\end{theorem}

 Our key ingredients are Lemma \ref{L3} and Lemma \ref{L4} which extend Proposition 2.1 in \cite{Chen.Ding}. Let us show that Theorem \ref{T.Ch.Ding} follows from Theorem \ref{T1}.
 \begin{proof}[Deduction of Theorem \ref{T.Ch.Ding} from Theorem \ref{T1}] Suppose that all assumptions of Theorem \ref{T.Ch.Ding} hold. Let $C$, $c$, $c_{0}$, and $c_{1}$ be the constants in Theorem \ref{T1}. Suppose that $x \geq x_{0}$, where $x_{0}$ is a constant depending only on $\alpha$ and $\beta$. This constant is large enough; it will be chosen later. We are going to deduce Theorem \ref{T.Ch.Ding} with $C_{0}=3C$. We note that we prove more than Theorem \ref{T.Ch.Ding} states, because $C_{0}$ is a positive absolute constant. We may assume that
 \[
 C< (\ln x)^{\alpha/2}< \frac{c (\ln x)^{1/4}}{(\ln\ln x)^{1/2}}
 \] for any $x\geq x_{0}$ (if $x_{0}$ is large enough). Let
 \[
 l= \min (k, \lfloor(\ln x)^{\alpha/2}\rfloor).
 \] Then
 \[
 C\leq l \leq (\ln x)^{\alpha/2}.
 \] Let $p_{0}$ be a prime number given by Theorem \ref{T1}. Since $P^{+}(a_{i})\leq 0.9 \ln\ln x$, we have $(a_i, p_0)=1$ for all $1 \leq i \leq k$. Also,
 \[a_{i}\leq (\ln x)^{\beta} \leq \textup{exp}(c_{0}\sqrt{\ln x}),\qquad
 b_{i} \leq x^{\alpha} \leq x\,\textup{exp}(c_0\sqrt{\ln x})
 \]for all $1\leq i \leq k$, if $x_{0}$ is large enough.

 Let $L_{i}(n)=a_{i}n+b_{i},$ $i=1,\ldots, k$. By Theorem \ref{T1} we have
 \begin{equation}\label{Ref_1}
 \#\{x\leq n< 2x: \#(\{L_{1}(n),\ldots, L_{l}(n)\}\cap \mathbb{P}) \geq C^{-1}\ln l\}\geq \frac{x}{(\ln x)^{l}\textup{exp}(Cl)}.
 \end{equation}We have
 \begin{align}
 (\ln x)^{l}\textup{exp}(Cl)&=\textup{exp}(l \ln\ln x + Cl)\notag\\
 &\leq \textup{exp}( (\ln x)^{\alpha/2}\ln\ln x + C(\ln x)^{\alpha/2})\leq \textup{exp}(C_{0}(\ln x)^{\alpha}),\label{Ref_2}
 \end{align}if $x_{0}$ is large enough.

 Let us show that
 \begin{equation}\label{ln_l_vs_ln_k}
 \ln l \geq \frac{\ln k}{3}.
 \end{equation} Indeed, if $C_{0}\leq k \leq \lfloor(\ln x)^{\alpha/2}\rfloor$, then $l=k$. Suppose that $\lfloor(\ln x)^{\alpha/2}\rfloor < k\leq (\ln x)^{\alpha}$. Then $l= \lfloor(\ln x)^{\alpha/2}\rfloor $. We have
 \[
 \frac{\alpha}{3}\ln\ln x\leq \ln k \leq \alpha \ln\ln x,\qquad \ln l\geq \frac{\alpha}{3}\ln\ln x,
 \]and \eqref{ln_l_vs_ln_k} is proved. Recall that $C_{0}=3C$. From \eqref{ln_l_vs_ln_k} we obtain
 \begin{align}
 \#\{x\leq n< 2x&: \#(\{L_{1}(n),\ldots, L_{l}(n)\}\cap \mathbb{P}) \geq C^{-1}\ln l\}\notag\\
 &\leq\#\{x\leq n< 2x: \#(\{L_{1}(n),\ldots, L_{l}(n)\}\cap \mathbb{P}) \geq C_{0}^{-1}\ln k\}\notag\\
 &\leq \#\{x\leq n< 2x: \#(\{L_{1}(n),\ldots, L_{k}(n)\}\cap \mathbb{P}) \geq C_{0}^{-1}\ln k\}.\label{Ref_3}
 \end{align}From \eqref{Ref_1}, \eqref{Ref_2} and \eqref{Ref_3} we get
 \[
 \#\{x\leq n< 2x: \#(\{L_{1}(n),\ldots, L_{k}(n)\}\cap \mathbb{P}) \geq C_{0}^{-1}\ln k\}\geq \frac{x}{\textup{exp}(C_{0}(\ln x)^{\alpha})}.
 \]This completes the proof of Theorem \ref{T.Ch.Ding} assuming Theorem \ref{T1}.
 \end{proof}

Also, we prove the following result.
\begin{theorem}\label{T2}
There are positive absolute constants $C$, $c$, $c_{1}$, and $c_2$ such that the following holds. Let $x\in \mathbb{R}$ and $k\in \mathbb{N}$ be such that $x\geq 10$, $k \geq C$ and $k \leq c(\ln x)^{1/4}(\ln\ln x)^{-1/2}$. Then there is a prime $p_{0}$ such that
\[
0.9\ln\ln x< p_{0}\leq \textup{exp} (c_1 \sqrt{\ln x})
\]and the following holds. Let $\mathcal{L}=\{L_{1},\ldots, L_{k}\}$ be a set of distinct linear functions $L_{i}(n)=a_{i}n-b_{i},$ $i=1,\ldots, k,$ with coefficients in the integers. Let the coefficients $L_{i}(n)=a_{i}n-b_{i}\in \mathcal{L}$ satisfy $(a_{i}, b_{i})=1,$
 \[
  1\leq a_{i}\leq \textup{exp}(c_{2}\sqrt{\ln x}),\qquad (a_{i}, p_0)=1,\qquad 1 \leq b_{i}\leq x
 \]for all $1\leq i \leq k$. Then
 \[
 \#\{x\leq n< 3x: \#(\{L_{1}(n),\ldots, L_{k}(n)\}\cap \mathbb{P}) \geq C^{-1}\ln k\}\geq \frac{x}{(\ln x)^{k}\textup{exp}(Ck)}.
 \]
\end{theorem}

Theorem \ref{T2} extends Theorem 1.1 in \cite{Chen.Ding} which showed the same result but with $a_{i}\leq 0.9\ln\ln x$, $b_{1}<\ldots<b_{k}\leq x$, and $k> (\ln x)^{\alpha}$, where $0<\alpha< 1/4$.

\begin{corollary}\label{C1}
There are positive absolute constants $C$ and $c$ such that the following holds. Let $x\in \mathbb{R}$ and $k\in \mathbb{N}$ be such that $x\geq 10$, $k \geq C$ and $k \leq c(\ln x)^{1/4}(\ln\ln x)^{-1/2}$. Let $a_{1},\ldots, a_{k}$ be distinct integers in the interval $[1,x]$. Then
 \[
 \#\{x\leq n< 3x: \#(\{n-a_{1},\ldots, n-a_{k}\}\cap \mathbb{P}) \geq C^{-1}\ln k\}\geq \frac{x}{(\ln x)^{k}\textup{exp}(Ck)}.
 \]
\end{corollary}

Corollary \ref{C1} is related to Romanoff's theorem. Romanoff \cite{Romanoff} proved that for any integer $a>1$ the integers of the form $p+a^{k}$ have positive density ($p$ runs through primes, $k$ runs through positive integers). In other words, the number of integers $\leq x$ of the form $p+a^{k}$ is greater than $c(a)x$, where $c(a)$ is a positive constant depending only on $a$. Romanoff proved in fact the following result. Denote by $f(n)$ the number of solutions of $p+a^{k}=n$. Then
\begin{equation}\label{I:Rom.Ineq}
\sum_{n\leq x} f^{2}(n) \leq c_{2}(a)x.
\end{equation}The fact that the numbers of the form $p+a^{k}$ have positive density follows immediately from \eqref{I:Rom.Ineq}, the Cauchy-Schwarz inequality and the inequality
\[
\sum_{n\leq x} f(n) \geq c_{3}(a)x.
\] Given a set of positive integers $\mathcal{A}$, we can consider the Romanoff type representation function
\[
f_{\mathcal{A}}(n)=\#\{a\in \mathcal{A}: n-a \in \mathbb{P}\}.
\] Corollary \ref{C1} provides a quantitative result for such a function. Erd\H{o}s \cite{Erdos} made the following conjecture: if $\#(\mathcal{A}\cap [1,x])>\ln x$ for all sufficiently large $x$, then $\lim\sup f_{\mathcal{A}}(n)=\infty$. This conjecture was proved in \cite{Chen.Ding1}. Corollary \ref{C1} shows that $\lim \sup f_{\mathcal{A}}(n)=\infty$ provided that $\#(\mathcal{A}\cap [1,x])\to \infty$ as $x\to \infty$. It is easy to see that Corollary \ref{C1} extends Corollary 1.3 and Corollary 1.4 in \cite{Chen.Ding}. We emphasize that the parameter $k$ in Corollary \ref{C1} can be small compared with $(\ln x)^{\alpha}$, where $\alpha$ is a positive constant. In particular, if we put
\[
\mathcal{A}=\big\{2^{2^{n}}: n\in \mathbb{N}\big\}
\](which corresponds to the case when the parameter $k$ is of the order of $\ln\ln x$), then by Corollary \ref{C1} for all sufficiently large $x$ we have
\[
\#\{x\leq n< 3x: f_{\mathcal{A}}(n)\geq c_{1} \ln\ln\ln x\}\geq \frac{x}{\textup{exp}(c_{2}(\ln\ln x)^{2})},
\]where $c_{1}$ and $c_{2}$ are positive absolute constants.

\section{Notation}
We reserve the letter $p$ for primes. In particular, the sum $\sum_{p\leq K}$ should be interpreted as being over all prime numbers not exceeding $K$. Let $\#\mathcal{A}$ denote the number of elements of a finite set $\mathcal{A}$. By $\mathbb{Z}$, $\mathbb{N}$, and $\mathbb{R}$ we denote the sets of all integers, positive integers, and real numbers respectively. By $\mathbb{P}$ we denote the set of all prime numbers.

 Let $(a,b)$ be the greatest common divisor of integers $a$ and $b$, and $[a,b]$ be the least common multiple of integers $a$ and $b$. If $d$ is a divisor of $b-a$, we say that $b$ is congruent to $a$ modulo $d$, and write $b \equiv a$ (mod $d$). We denote the number of primes not exceeding $x$ by $\pi (x)$ and the number of primes not exceeding $x$ that are congruent to $a$ modulo $q$ by $\pi (x; q,a)$. Let
  \[
  \textup{li}(x):=\int_{2}^{x}\frac{dt}{\ln t}.
  \]Let $\varphi$ denote Euler's totient function, i.\,e. $\varphi(n)=\#\{1\leq m \leq n:\ (m,n)=1\}$. The symbol $b|a$ means that $b$ divides $a$. For a real number $x$ by $\lfloor x\rfloor$ we denote the largest integer not exceeding $x$, and by $\lceil x\rceil$ we denote the smallest integer $n$ such that $n\geq x$.

For real numbers $x,$ $y$ we also use $(x,y)$ to denote the open interval, and $[x,y]$ to denote the closed interval. The usage of the notation will be clear from the context.

 We shall view $0<\theta<1$ as a fixed real constant. Notation $A\ll B$ means that $|A|\leq c B$ with a positive constant which may depend on $\theta$ but no other variable, unless otherwise noted.

 \section{Well-distributed sets}

 Given a set of integers $\mathcal{A}$, a set of primes $\mathcal{P}$, and a linear function $L(n)=l_{1} n+l_{2}$, we define
 \begin{align*}
 \mathcal{A}(x)&=\{n\in \mathcal{A}: x\leq n< 2x\},&\qquad \mathcal{A}(x;q,a)&=\{n\in \mathcal{A}(x): n\equiv a\text{ (mod $q$)}\},\\
 L(\mathcal{A})&=\{L(n): n\in \mathcal{A}\},&\qquad \varphi_{L}(q)&=\varphi(|l_{1}|q)/\varphi(|l_{1}|),\\
 \mathcal{P}_{L,\mathcal{A}}(x)&=L(\mathcal{A}(x))\cap \mathcal{P},&\quad \mathcal{P}_{L,\mathcal{A}}(x;q,a)&=L(\mathcal{A}(x;q,a))\cap \mathcal{P}.
 \end{align*}

 We will focus on sets which satisfy the following hypothesis, which is given in terms of $(\mathcal{A}, \mathcal{L}, \mathcal{P}, B, x, \theta)$ for $\mathcal{L}$ a set of linear functions, $B\in \mathbb{N}$, $x$ a large real number, and $0<\theta <1$.

\begin{hypothesis}\label{Hypothesis_1}
$(\mathcal{A}, \mathcal{L}, \mathcal{P}, B, x, \theta)$. Let $k=\# \mathcal{L}$.\par\smallskip
\textup{(1)} $\mathcal{A}$ is well-distributed in arithmetic progressions\textup{:} we have
\begin{equation}\label{Hyp.1}
\sum_{1\leq q \leq x^{\theta}}\max_{a\in \mathbb{Z}}\biggl|\#\mathcal{A}(x; q, a) - \frac{\# \mathcal{A}(x)}{q}\biggr|\ll \frac{\#\mathcal{A}(x)}{(\ln x)^{100 k^{2}}}.
\end{equation}

\textup{(2)} Primes in $L(\mathcal{A})\cap \mathcal{P}$ are well-distributed in most arithmetic progressions\textup{:} for any $L\in \mathcal{L}$ we have
\begin{equation}\label{Hyp.2}
\sum_{\substack{1\leq q\leq x^{\theta}\\ (q, B)=1}} \max_{\substack{a\in \mathbb{Z}\\(L(a), q)=1}}\biggl| \#\mathcal{P}_{L,\mathcal{A}}(x; q, a) - \frac{\#\mathcal{P}_{L,\mathcal{A}}(x)}{\varphi_{L}(q)}\biggr|\ll \frac{\#\mathcal{P}_{L,\mathcal{A}}(x)}{(\ln x)^{100 k^{2}}}.
\end{equation}

\textup{(3)} $\mathcal{A}$ is not too concentrated in any arithmetic progression\textup{:} for any $1\leq q< x^{\theta}$ we have
\begin{equation}\label{Hyp.3}
\max_{a\in\mathbb{Z}}\#\mathcal{A}(x; q, a) \ll \frac{\#\mathcal{A}(x)}{q}.
\end{equation}
\end{hypothesis}

 \section{Preparatory Lemmas}

 \begin{lemma}\label{L2}
 There is a positive absolute constant $d$ such that the following holds. Let $k$ be a positive integer with $k\geq 2$. Let $\mathcal{L}=\{L_{1},\ldots, L_{k}\}$ be a set of distinct linear functions $L_{i}(n)=a_{i}n+b_{i},$ $i=1,\ldots, k,$ with coefficients in the positive integers. Let $(a_i, b_i)=1$ for any $1\leq i \leq k$. Then there is a sequence of positive integers $i_{1}<\ldots <i_{s}\leq k$ such that
 \[
 s= \bigg\lceil d\,\frac{k}{\ln k} \bigg\rceil
 \] and the set $\mathcal{L}^{\prime}=\{L_{i_{1}},\ldots, L_{i_{s}}\}$ is admissible.
 \end{lemma}

 \begin{proof} This is \cite[Lemma 2.7]{Chen.Ding}.
 \end{proof}

 \begin{lemma}\label{L3}
  Let $\varepsilon$ be a real number with $0 < \varepsilon < 1/2$. Then there is a positive constant $c(\varepsilon)$, depending only on $\varepsilon$, such that the following holds. Let $x$ be a real number with $x\geq c(\varepsilon)$ Then there is a prime $B$ such that
 \[
 0.9 \ln\ln x< B \leq \textup{exp} (c_1 \sqrt{\ln x})
\]and
\begin{equation}\label{L3:Part2}
\sum_{\substack{1\leq r\leq x^{1/2 - \varepsilon}\\
(r,B)=1}} \max_{2 \leq u \leq x^{1+ \gamma/\sqrt{\ln x}}}\max_{\substack{b\in \mathbb{Z}:\\ (b,r)=1}}
\biggl|\pi (u; r, b)-\frac{\textup{li}(u)}{\varphi (r)}\biggr| \leq c_2 x\,\textup{exp}(-c_3\sqrt{\ln x}).
\end{equation}Here $c_1$, $c_2$, $c_3$, and $\gamma$ are positive absolute constants.
 \end{lemma}

 \begin{proof}We assume that $c(\varepsilon)$ is large enough; this constant will be chosen later. By the Landau-Page theorem (see, for example, \cite[Chapter 14]{Davenport}) there is at most one modulus $q_{0}\in [3, \textup{exp}(2 c_{1} \sqrt{\ln x})]$ such that there exists a primitive character $\chi$ modulo $q_{0}$ for which $L(s,\chi)$ has a real zero larger than $1-c_{4}(\ln x)^{-1/2}$ (for suitable positive absolute constants $c_{1},$ $c_{4}$). If this exceptional modulus $q_{0}$ exists, we take $B$ to be the largest prime factor of $q_{0}$, and otherwise we take $B$ to be the smallest prime on the interval $(0.9 \ln\ln x, \textup{exp} (2c_1 \sqrt{\ln x})]$, which exists by Bertrand's postulate (see, for example, \cite[Theorem 3.1.9]{Murty}).

 Suppose that $q_{0}$ exists. We have (see, for example, \cite[Chapter 20]{Davenport})
 \begin{equation}\label{L3:q0.grow}
 q_{0}\geq c_{5} \frac{\ln x}{(\ln\ln x)^{4}},
 \end{equation}where $c_{5}$ is a positive absolute constant (provided that $c(\varepsilon)$ is large enough). Hence, we may assume that $q_{0}\geq 100$. It is well-known, the number $q_{0}$ is of the form $2^{\alpha} q_{1}$, where $\alpha \in \{0,\ldots, 3\}$ and $q_{1}\geq 3$ is an odd square-free integer (see, for example, \cite[Lemma 4.5]{Radomskii}). Let us show that
 \begin{equation}\label{L3:B.grow}
 B>0.9\ln\ln x.
 \end{equation}Assume the converse $B\leq 0.9\ln\ln x$. Hence,
 \[
 q_{0}\leq 4\prod_{p\leq 0.9\ln\ln x} p.
 \]We have
 \[
 \sum_{p\leq 0.9\ln\ln x}\ln p\leq 0.99 \ln\ln x,
 \]provided that $c(\varepsilon)$ is large enough. We obtain
 \[
 q_{0}\leq 4 (\ln x)^{0.99},
 \]but this contradicts \eqref{L3:q0.grow}. Hence, \eqref{L3:B.grow} is proved. Also,
 \[
 B\leq q_{0} \leq \textup{exp}(2 c_{1} \sqrt{\ln x}).
 \] The proof of \eqref{L3:Part2} repeats the proof of Lemma 4.9 and Lemma 4.10 in \cite{Radomskii}.
 \end{proof}

 \begin{lemma}\label{L4}
 Let $\mathcal{A}= \mathbb{N}$, $\mathcal{P}=\mathbb{P}$, $\theta=1/3$. Let $x\in \mathbb{N}$, $k\in \mathbb{N}$ be such that $x \geq x_{0}$, $k\leq \eta(\ln x)^{1/4}(\ln\ln x)^{-1/2}$. Then there is a prime $B$  such that
 \begin{equation}\label{L4:B.range}
 0.9 \ln\ln x< B \leq \textup{exp} (c_1 \sqrt{\ln x})
 \end{equation}and the following holds. Let $\mathcal{L}=\{L_{1},\ldots, L_{k}\}$ be an admissible set of $k$ linear functions. Let the coefficients $L_{i}(n)=a_{i}n+b_{i}\in \mathcal{L}$ satisfy
 \[
  1\leq a_{i}\leq \textup{exp}(c_{0}\sqrt{\ln x}),\qquad (a_{i},B)=1,\qquad 1\leq b_{i}\leq x\,\textup{exp}(c_{0}\sqrt{\ln x})
 \]for all $1\leq i \leq k$. Then $(\mathcal{A}, \mathcal{L}, \mathcal{P}, B, x, \theta)$ satisfy Hypothesis \ref{Hypothesis_1}; the implied constants in \eqref{Hyp.1}--\eqref{Hyp.3} are positive and absolute. Here $x_{0},$ $c_{0}$, $c_{1}$, and $\eta$ are positive absolute constants.
 \end{lemma}

 \begin{proof} We assume that $x_{0}$ is large enough and $c_{0}$ is small enough; we will choose the constants $x_{0}$ and $c_{0}$ later. Let $c(\varepsilon)$ be the constant in Lemma \ref{L3}. We take $\varepsilon = 0.1$, and let  $x_{0}\geq c(0.1)$. By Lemma \ref{L3}, there is a prime $B$ such that
 \[
 0.9 \ln\ln x< B \leq \textup{exp} (c_1 \sqrt{\ln x})
 \] and
\begin{equation}\label{L4:Bombieri-Vinogradov}
\sum_{\substack{1\leq r\leq x^{0.4}\\
(r,B)=1}} \max_{2 \leq u \leq x^{1+ \gamma/\sqrt{\ln x}}}\max_{\substack{b\in \mathbb{Z}:\\ (b,r)=1}}
\biggl|\pi (u; r, b)-\frac{\textup{li}(u)}{\varphi (r)}\biggr| \leq c_2 x\,\textup{exp}(-c_3\sqrt{\ln x}),
\end{equation}where $c_1$, $c_2$, $c_3$, and $\gamma$ are positive absolute constants. We assume that $c_{0}\leq \gamma / 2$. Let us show that $(\mathcal{A}, \mathcal{L}, \mathcal{P}, B, x, \theta)$ satisfy Hypothesis \ref{Hypothesis_1}.

I) Let $L(n)=l_{1} n+ l_{2}\in \mathcal{L}$. Let us show that
\begin{equation}\label{L4:Basic.I}
S:=\sum_{\substack{1 \leq r \leq x^{1/3}\\ (r, B)=1}}\max_{\substack{b\in \mathbb{Z}\\ (L(b), r)=1}}
\biggl|\#\mathcal{P}_{L, \mathcal{A}}(x; r, b) - \frac{\#\mathcal{P}_{L, \mathcal{A}}(x)}{\varphi_{L}(r)}\biggr|
\leq \frac{\#\mathcal{P}_{L, \mathcal{A}}(x)}{(\ln x)^{100 k^{2}}}.
\end{equation}It is easy to see that
\begin{gather*}
\mathcal{P}_{L, \mathcal{A}}(x)=\{ l_{1}x+ l_{2} \leq p < 2l_{1}x+ l_{2}: p\equiv l_{2}\ \text{(mod $l_{1}$)}\},\\
\mathcal{P}_{L, \mathcal{A}}(x; r, b)=\{l_{1}x + l_{2}\leq p < 2l_{1}x+l_{2}: p\equiv l_{1}b+l_{2}\ \text{(mod $l_{1}r$)}\},
\end{gather*}and hence
\begin{gather*}
\#\mathcal{P}_{L, \mathcal{A}}(x)= \pi(2l_{1}x+l_{2}-1; l_1, l_2) -
\pi(l_{1}x+l_{2}-1; l_1, l_2),\label{Number_P_L_A_for_T6}\\
\#\mathcal{P}_{L, \mathcal{A}}(x; r, b)=\pi(2l_{1}x+l_{2}-1; l_{1}r, L(b))-
\pi(l_{1}x+l_{2}-1; l_{1}r, L(b)).\notag
\end{gather*}We obtain
\begin{align*}
S=\sum_{\substack{1 \leq r \leq x^{1/3}\\ (r, B)=1}}&\max_{\substack{b\in \mathbb{Z}\\ (L(b), r)=1}}
\biggl|\pi(2l_{1}x+l_{2}-1; l_{1}r, L(b))-
\pi(l_{1}x+l_{2}-1; l_{1}r, L(b)) \\
&- \frac{\pi(2l_{1}x+l_{2}-1; l_1, l_2) -
\pi(l_{1}x+l_{2}-1; l_1, l_2)}{\varphi(l_{1}r)/\varphi(l_{1})}\biggr|\leq
 S_{1}+S_{2}+S_{3}+S_{4},
\end{align*}where
\begin{align*}
S_{1} &= \sum_{\substack{1 \leq r \leq x^{1/3}\\ (r, B)=1}}\max_{\substack{b\in \mathbb{Z}\\ (L(b), r)=1}}
\biggl| \pi(l_{1}x+l_{2}-1; l_{1}r, L(b)) - \frac{\textup{li}(l_{1}x+l_{2}-1)}{\varphi(l_{1}r)}\biggr|,\\
S_{2} &= \sum_{\substack{1 \leq r \leq x^{1/3}\\ (r, B)=1}}
\biggl|\frac{\pi(l_{1}x+l_{2}-1; l_{1}, l_{2})}{\varphi(l_{1}r)/\varphi(l_{1})} -
\frac{\textup{li}(l_{1}x+l_{2}-1)}{\varphi(l_{1}r)} \biggr|,\\
S_{3} &= \sum_{\substack{1 \leq r \leq x^{1/3}\\ (r, B)=1}}\max_{\substack{b\in \mathbb{Z}\\ (L(b), r)=1}}
\biggl| \pi(2l_{1}x+l_{2}-1; l_{1}r, L(b)) - \frac{\textup{li}(2l_{1}x+l_{2}-1)}{\varphi(l_{1}r)}\biggr|,\\
S_{4} &= \sum_{\substack{1 \leq r \leq x^{1/3}\\ (r, B)=1}}
\biggl|\frac{\pi(2l_{1}x+l_{2}-1; l_{1}, l_{2})}{\varphi(l_{1}r)/\varphi(l_{1})} -
\frac{\textup{li}(2l_{1}x+l_{2}-1)}{\varphi(l_{1}r)} \biggr|.
\end{align*}

Let us show that
\begin{equation}\label{L(b)_coprime_l1}
(L(b), l_{1})=1
\end{equation}for any $b\in \mathbb{Z}$. Assume the converse: there is an integer $b$ such that $(L(b), l_{1})>1$. Then there is a prime $p$ such that $p|l_{1}$ and $p|L(b)$. Hence $p|l_{2}$, and we see that $p|L(n)$ for any integer $n$. Since $L\in \mathcal{L}$, we see that $p|L_{1}(n)\cdots L_{k}(n)$ for any integer $n$. But this contradicts the fact that $\mathcal{L}=\{L_{1},\ldots, L_{k}\}$ is an admissible set. Thus, \eqref{L(b)_coprime_l1} is proved.

 We observe that since $l_1= a_{i}$ for some $1\leq i \leq k$, we have
\begin{equation}\label{L4:B.coprime.l1}
(B,l_{1})=1.
\end{equation}We have
\begin{gather}
l_{1} x^{1/3}\leq  x^{1/3} \textup{exp}(c_{0}\sqrt{\ln x})\leq x^{1/3} \textup{exp}((\gamma/2)\sqrt{\ln x})  \leq x^{0.4},\notag\\
l_{1}x+l_{2}-1\geq l_{1}x\geq x\geq 2,\notag\\
2 l_{1}x+ l_{2} -1 \leq 3 x\,\textup{exp}(c_{0}\sqrt{\ln x}) \leq 3 x\,\textup{exp}((\gamma/2)\sqrt{\ln x}) \leq x^{1+\gamma/\sqrt{\ln x}},\label{L4:2l1.up}
\end{gather}if $x_{0}$ is large enough. We obtain (see \eqref{L(b)_coprime_l1}, \eqref{L4:B.coprime.l1} and \eqref{L4:Bombieri-Vinogradov})
\begin{align*}
S_{1} &= \sum_{\substack{r:\\ l_{1} \leq l_{1}r \leq l_{1}x^{1/3}\\ (l_{1}r, B)=1}}\max_{\substack{b\in \mathbb{Z}\\ (L(b), l_{1}r)=1}}
\biggl| \pi(l_{1}x+l_{2}-1; l_{1}r, L(b)) - \frac{\textup{li}(l_{1}x+l_{2}-1)}{\varphi(l_{1}r)}\biggr|\\
&\leq \sum_{\substack{1\leq Q\leq x^{0.4}\\
(Q,B)=1}} \max_{2 \leq u \leq x^{1+ \gamma/\sqrt{\ln x}}}\max_{\substack{W\in \mathbb{Z}:\\ (W,Q)=1}}
\biggl|\pi (u; Q, W)-\frac{\textup{li}(u)}{\varphi (Q)}\biggr| \leq c_2 x\,\textup{exp}(-c_3\sqrt{\ln x}).
\end{align*}It is well-known (see, for example, \cite[Chapter 28]{Davenport}),
 \[
 \sum_{n\leq x}\frac{1}{\varphi(n)}\leq c \ln x
 \]for any real $x\geq 2$, where $c$ is a positive absolute constant. Since $\varphi(mn)\geq \varphi (m) \varphi (n)$ for any positive integers $m$ and $n$, we obtain
\begin{align*}
S_{2}&=\varphi(l_{1})\biggl|\pi(l_{1}x+l_{2}-1; l_{1}, l_{2})- \frac{\textup{li}(l_{1}x+l_{2}-1)}{\varphi(l_{1})}\biggr|
\sum_{\substack{1 \leq r \leq x^{1/3}\\ (r, B)=1}}\frac{1}{\varphi (l_{1}r)}\\
&\leq \biggl|\pi(l_{1}x+l_{2}-1; l_{1}, l_{2})- \frac{\textup{li}(l_{1}x+l_{2}-1)}{\varphi(l_{1})}\biggr|
\sum_{1 \leq r \leq x^{1/3}}\frac{1}{\varphi (r)}\\
&\leq \widetilde{c}\ln x \biggl|\pi(l_{1}x+l_{2}-1; l_{1}, l_{2})- \frac{\textup{li}(l_{1}x+l_{2}-1)}{\varphi(l_{1})}\biggr|,
\end{align*}where $\widetilde{c}>0$ is an absolute constant. Since $1\leq l_{1} \leq x^{0.4},$ $(l_{1}, B)=1,$ $(l_{1}, l_{2})=1,$ and
\[
2\leq l_{1}x+l_{2}-1 \leq x^{1+\gamma/\sqrt{\ln x}},
\] from \eqref{L4:Bombieri-Vinogradov} we obtain
\begin{equation}\label{L4:PNT.l1.x}
\biggl|\pi(l_{1}x+l_{2}-1; l_{1}, l_{2})- \frac{\textup{li}(l_{1}x+l_{2}-1)}{\varphi(l_{1})}\biggr|
\leq c_2 x\,\textup{exp}(-c_3\sqrt{\ln x}).
\end{equation}Hence,
\[
S_{2}\leq c_{4}  x\,\textup{exp}(-c_3\sqrt{\ln x}+\ln\ln x)\leq
c_{4}  x\,\textup{exp}(-(c_3/2)\sqrt{\ln x}),
\]provided that $x_{0}$ is large enough (here $c_{4}= \widetilde{c} c_{2}$ is a positive absolute constant).

 Similarly, it can be shown that
\[
S_{3} \leq C x\,\textup{exp}(-c\sqrt{\ln x}),\qquad S_{4} \leq C x\,\textup{exp}(-c\sqrt{\ln x}),
\]where $C$ and $c$ are positive absolute constants. We obtain
\begin{equation}\label{L4:Summa.P.L.A}
\sum_{\substack{1 \leq r \leq x^{1/3}\\ (r, B)=1}}\max_{\substack{b\in \mathbb{Z}\\ (L(b), r)=1}}
\biggl|\#\mathcal{P}_{L, \mathcal{A}}(x; r, b) - \frac{\#\mathcal{P}_{L, \mathcal{A}}(x)}{\varphi_{L}(r)}\biggr|
\leq c_{5} x\,\textup{exp}(-c_{6}\sqrt{\ln x}),
\end{equation}where $c_{5}$ and $c_{6}$ are positive absolute constants.

 By \eqref{L4:PNT.l1.x}, we have
\[
\pi (l_{1}x+l_{2}-1; l_{1}, l_{2}) = \frac{\textup{li}(l_{1}x+l_{2}-1)}{\varphi (l_{1})} + R_{1},\qquad
|R_{1}| \leq c_2 x\,\textup{exp}(-c_3\sqrt{\ln x}).
\]Similarly, it can be shown that
\[
\pi (2l_{1}x+l_{2}-1; l_{1}, l_{2}) = \frac{\textup{li}(2l_{1}x+l_{2}-1)}{\varphi (l_{1})} + R_{2},\qquad
|R_{2}| \leq c_2 x\,\textup{exp}(-c_3\sqrt{\ln x}).
\] We obtain
\begin{equation}\label{L4:P.L.A.main.term}
\#\mathcal{P}_{L, \mathcal{A}}(x) = \frac{\textup{li}(2l_{1}x+l_{2}-1) - \textup{li}(l_{1}x+l_{2}-1)}{\varphi (l_{1})}+ R,\qquad
 |R|\leq c_{7} x\,\textup{exp}(-c_{3}\sqrt{\ln x}),
\end{equation}where $c_{7}=2c_{2}$ is a positive absolute constant. By \eqref{L4:2l1.up}, we have
\[
\ln (2l_{1} x+ l_{2} -1) \leq \ln x + \gamma \sqrt{\ln x}\leq 2 \ln x,
\] if $x_{0}$ is large enough. Hence,
\begin{align*}
\frac{\textup{li}(2l_{1}x+l_{2}-1)-\textup{li}(l_{1}x+l_{2}-1)}{\varphi(l_{1})}&=\frac{1}{\varphi(l_{1})}\int_{l_{1}x+l_{2}-1}^{2l_{1}x+l_{2}-1}
\frac{dt}{\ln t}\\
&\geq \frac{l_{1} x}{\varphi(l_{1})\ln (2l_{1}x+l_{2}-1)} \geq
\frac{l_{1} x}{2\varphi(l_{1})\ln x} .
\end{align*}Let us show that
\begin{equation}\label{L4:|R|<x/ln_x}
|R| \leq \frac{l_{1} x}{4\varphi(l_{1})\ln x}.
\end{equation}Since $l_{1}/\varphi (l_{1})\geq 1$, we see from \eqref{L4:P.L.A.main.term} that it is sufficient to show that   \[
c_{7} x\,\textup{exp}(-c_{3}\sqrt{\ln x}) \leq \frac{ x}{4\ln x}.
\]This inequality holds, if $x_{0}$ is large enough. Thus, \eqref{L4:|R|<x/ln_x} is proved. We obtain
\begin{equation}\label{L4:P.L.A.Low.bound}
\#\mathcal{P}_{L, \mathcal{A}}(x) \geq \frac{l_{1} x}{4\varphi(l_{1})\ln x}.
\end{equation}

Now we prove \eqref{L4:Basic.I}. Since $l_{1}/\varphi (l_{1})\geq 1$, we see from \eqref{L4:Summa.P.L.A} and \eqref{L4:P.L.A.Low.bound} that it suffices to show that
\[
c_{5} x\,\textup{exp}(-c_{6}\sqrt{\ln x})\leq \frac{x}{4 (\ln x)^{100 k^{2}+1}}.
\] But this inequality holds, if $\eta$ is small enough and $x_{0}$ is large enough (we recall that $k\leq \eta (\ln x)^{1/4}(\ln\ln x)^{-1/2}$). Therefore \eqref{L4:Basic.I} is proved.

II) Let us show that
\begin{equation}\label{Basic_II_for_L4}
\sum_{1 \leq r \leq x^{1/3}}\max_{b\in \mathbb{Z}}\biggl|\#\mathcal{A}(x; r, b) - \frac{\# \mathcal{A}(x)}{r}\biggr|\leq \frac{\#\mathcal{A}(x)}{(\ln x)^{100 k^{2}}}.
\end{equation}Let $1 \leq r \leq x^{1/3}$, $b\in \mathbb{Z}$. We have
\[
\mathcal{A}(x)=\{x\leq n < 2x\},\qquad
\mathcal{A}(x; r, b) = \{x \leq n < 2x:\ n \equiv b\ \text{(mod $r$)}\}.
\]Hence,
\begin{equation}\label{L4:A.view}
\#\mathcal{A}(x)=x;\qquad
\#\mathcal{A}(x; r, b)=\frac{x}{r}+\rho,\quad |\rho|\leq 1.
\end{equation}We obtain
\[
\biggl|\#\mathcal{A}(x; r, b) - \frac{\# \mathcal{A}(x)}{r}\biggr|=
|\rho| \leq 1.
\]Hence,
\[
\sum_{1 \leq r \leq x^{1/3}}\max_{b\in \mathbb{Z}}\biggl|\#\mathcal{A}(x; r, b) - \frac{\# \mathcal{A}(x)}{r}\biggr| \leq  x^{1/3}.
\]Thus, to prove \eqref{Basic_II_for_L4} it suffices to show that
\[
x^{1/3}\leq \frac{x}{(\ln x)^{100 k^{2}}}.
\]But this inequality holds, if $\eta$ is small enough and $x_{0}$ is large enough. Thus, \eqref{Basic_II_for_L4} is proved.

III) Let us show that for any integer $r$ with $1\leq r < x^{1/3}$ we have
\begin{equation}\label{Basic_III_for_L4}
\max_{b\in\mathbb{Z}}\#\mathcal{A}(x; r, b) \leq 2\, \frac{\#\mathcal{A}(x)}{r}.
\end{equation}Let $1\leq r < x^{1/3}$ and $b\in \mathbb{Z}$. We may assume that $x_{0}\geq 2$. Hence, $r \leq x^{1/3} \leq x$. Applying \eqref{L4:A.view}, we obtain
\[
\#\mathcal{A}(x; r, b)\leq \frac{x}{r}+1\leq 2\,\frac{x}{r}= 2\,\frac{\#\mathcal{A}(x)}{r},
\]and \eqref{Basic_III_for_L4} is proved. Thus, $(\mathcal{A}, \mathcal{L}, \mathcal{P}, B, x, 1/3)$ satisfy Hypothesis \ref{Hypothesis_1}. Lemma \ref{L4} is proved.

 \end{proof}

 \begin{lemma}\label{L5}
 Let $0<\theta <1$. Let $\mathcal{A}$ be a set of positive integers, $\mathcal{P}$  a set of primes, $\mathcal{L}=\{L_{1},\ldots, L_{k}\}$ an admissible set of $k$ linear functions, and $B,$ $x$ positive integers. Let the coefficients $L_{i}(n)=a_{i} n + b_{i}\in \mathcal{L}$ satisfy $1\leq a_{i}, b_i\leq x^{6/5}$ for all $1\leq i \leq k$, and let $k\leq (\ln x)^{6/5}$ and $B\leq x^{6/5}$. Then there is a constant $C_{0}$, depending only on $\theta$, such that the following holds. If $k\geq C_{0}$ and $(\mathcal{A}, \mathcal{L}, \mathcal{P}, B, x, \theta)$ satisfy Hypothesis \ref{Hypothesis_1}, and if $\delta > (\ln k)^{-1}$ is such that
 \[
 \frac{1}{k} \frac{\varphi(B)}{B}\sum_{i=1}^{k} \frac{\varphi (a_{i})}{a_{i}} \#\mathcal{P}_{L_{i}, \mathcal{A}}(x)\geq
 \delta \frac{\#\mathcal{A}(x)}{\ln x},
 \]then
 \[
 \#\{n\in \mathcal{A}(x): \#(\{L_{1}(n),\ldots, L_{k}(n)\}\cap\mathcal{P})\geq C_{0}^{-1}\delta \ln k\} \gg
  \frac{\#\mathcal{A}(x)}{(\ln x)^{k}\textup{exp} (C_{0}k)}.
 \]
  \end{lemma}

  \begin{proof} This is \cite[Theorem 3.1]{Maynard} with $\alpha=6/5$.
  \end{proof}

  \begin{lemma}\label{L6}
  There are positive absolute constants $D$, $c_{0}$, $c_1$, and $\eta$ such that the following holds. Let $x\in \mathbb{R}$ and $k\in \mathbb{N}$ be such that $x\geq 10$, $k \geq D$ and $k \leq \eta(\ln x)^{1/4}(\ln\ln x)^{-1/2}$. Then there is a prime $B$ such that
  \[
  0.9 \ln\ln x< B \leq \textup{exp} (c_1 \sqrt{\ln x})
  \]and the following holds. Let $\mathcal{L}=\{L_{1},\ldots, L_{k}\}$ be an admissible set of $k$ linear functions. Let the coefficients $L_{i}(n)=a_{i} n + b_{i}\in \mathcal{L}$ satisfy
   \[
  1\leq a_{i}\leq \textup{exp}(c_{0}\sqrt{\ln x}),\qquad (a_{i},B)=1,\qquad 1\leq b_{i}\leq x\,\textup{exp}(c_{0}\sqrt{\ln x})
 \]for all $1\leq i \leq k$. Then
 \[
 \#\{ x\leq n \leq 2x-2: \#(\{L_{1}(n),\ldots, L_{k}(n)\}\cap \mathbb{P}) \geq D^{-1}\ln k\} \geq
 \frac{x}{(\ln x)^{k} \textup{exp}(Dk)}.
 \]
  \end{lemma}

 \begin{proof} Let $x_0$, $c_{0}$, $c_{1}$, and $\eta$ be the constants in Lemma \ref{L4}, and $C_{0}(\theta)$ be the constant in Lemma \ref{L5}. Let $\mathcal{A}= \mathbb{N}$, $\mathcal{P}=\mathbb{P}$, $\theta=1/3$, $C_{0}=C_{0}(1/3)$.  We assume that $D$ is large enough; we will choose the constant $D$ later. We assume that $D\geq C_{0}$. It is clear that $x$ is large, if $D$ is large.

  Suppose that $x$ is a positive integer. We have $x\geq x_{0}$, if $D$ is large enough. By Lemma \ref{L4}, there is a prime $B$ such that \eqref{L4:B.range} holds and $(\mathcal{A}, \mathcal{L}, \mathcal{P}, B, x, 1/3)$ satisfy Hypothesis \ref{Hypothesis_1}. We have $\textup{exp}(c_{0}\sqrt{\ln x}) \leq x^{1/5}$ and $\textup{exp}(c_{1}\sqrt{\ln x}) \leq x^{6/5} $, if $D$ is large enough. We obtain $B \leq x^{6/5}$, and $1\leq a_{i}, b_{i} \leq x^{6/5}$ for all $1\leq i \leq k$. By \eqref{L4:P.L.A.Low.bound}, we have
\[
\#\mathcal{P}_{L_{i}, \mathcal{A}}(x) \geq \frac{a_{i}x}{4\varphi(a_{i})\ln x}.
\]Also,
\[
\frac{B}{\varphi(B)}=\frac{B}{B-1}\leq 2.
\]We obtain
\[
 \frac{1}{k} \frac{\varphi(B)}{B}\sum_{i=1}^{k} \frac{\varphi (a_{i})}{a_{i}} \#\mathcal{P}_{L_{i}, \mathcal{A}}(x)\geq
   \frac{x}{8\ln x}=
 \delta \frac{\#\mathcal{A}(x)}{\ln x},
 \]where $\delta = 1/8$. We have $\delta > (\ln k)^{-1}$, if $D$ is large enough. By Lemma \ref{L5}, we have
 \[
 \#\{x\leq n \leq 2x-1: \#(\{L_{1}(n),\ldots, L_{k}(n)\}\cap\mathbb{P})\geq (8C_{0})^{-1} \ln k\} \geq c_{2}
  \frac{x}{(\ln x)^{k}\textup{exp} (C_{0}k)},
 \]where $c_{2}$ is a positive absolute constant.

 We put
 \[
 N_{i}=\#\{x\leq n \leq 2x-(3i-2):\,\#(\{L_{1}(n),\ldots, L_{k}(n)\}\cap\mathbb{P})\geq (8 C_{0})^{-1} \ln k\},\quad i=1, 2.
 \]It is clear that $N_{1}\leq N_{2}+3$ and
 \[
  \frac{c_{2}}{2}\frac{x}{(\ln x)^{k}\textup{exp} (C_{0}k)}\geq 3,
 \] if $D$ is large enough. Hence,
 \begin{align}
 \#\{x\leq n \leq 2x-4:\, \#(\{&L_{1}(n),\ldots, L_{k}(n)\}\cap\mathbb{P})\geq (8 C_{0})^{-1} \ln k\}
 \geq N_{1}-3\notag\\
 &\geq c_{2}\frac{x}{(\ln x)^{k}\textup{exp} (C_{0} k)} -3\geq  \frac{c_{2}}{2}\frac{x}{(\ln x)^{k}\textup{exp} (C_{0} k)}.\label{L6:Basic.integer}
 \end{align}

 Suppose that $x\in \mathbb{R}$, and let $l=\lceil x \rceil$. Then $l$ is a positive integer with $l\geq 10$ and
 \[
 k\leq \eta\frac{(\ln x)^{1/4}}{(\ln\ln x)^{1/2}}\leq \eta\frac{(\ln l)^{1/4}}{(\ln\ln l)^{1/2}}.
 \] We have
 \[
 0.9\ln\ln x\leq 0.9\ln\ln l < B \leq \textup{exp} (c_1 \sqrt{\ln l})\leq \textup{exp} (2c_1 \sqrt{\ln x})
 \]and
 \begin{gather*}
  1\leq a_{i}\leq \textup{exp}(c_{0}\sqrt{\ln x})\leq \textup{exp}(c_{0}\sqrt{\ln l}),\qquad (a_{i},B)=1,\\ 1 \leq b_{i}\leq x\,\textup{exp}(c_{0}\sqrt{\ln x})  \leq l\,\textup{exp}(c_{0}\sqrt{\ln l})
 \end{gather*}for all $1\leq i \leq k$. By \eqref{L6:Basic.integer}, we have
 \[
 \#\{l\leq n \leq 2l-4:\, \#(\{L_{1}(n),\ldots, L_{k}(n)\}\cap\mathbb{P})\geq (8 C_{0})^{-1} \ln k\}\geq \frac{c_{2}}{2}
 \frac{l}{(\ln l)^{k}\textup{exp} (C_{0} k)}.
  \] We have
  \[
  \frac{2}{c_{2}}(\ln l)^{k}\textup{exp} (C_{0} k)\leq \frac{2}{c_{2}}(\ln (x+1))^{k}\textup{exp} (C_{0} k)\leq \frac{2}{c_{2}}(2\ln x)^{k}
  \textup{exp} (C_{0} k)\leq (\ln x)^{k} \textup{exp} (D k),
    \] if $D$ is large enough. Hence,
    \[
    \frac{c_{2}}{2}\frac{l}{(\ln l)^{k}\textup{exp} (C_{0} k)} \geq \frac{x}{(\ln x)^{k}\textup{exp} (D k)}.
    \]We may assume that $D \geq 8 C_{0}$. Since $x \leq l < x+1$, we obtain
    \[
    \#\{x\leq n \leq 2x-2:\, \#(\{L_{1}(n),\ldots, L_{k}(n)\}\cap\mathbb{P})\geq D^{-1} \ln k\}\geq
 \frac{x}{(\ln x)^{k}\textup{exp} (D k)}.
    \]Finally, let us denote $2c_1$ by $c_1$. Lemma \ref{L6} is proved.

 \end{proof}

 \section{Proof of Theorem \ref{T1} and Theorem \ref{T2}}

 \begin{proof}[Proof of Theorem \ref{T1}.] Let $D$, $c_{0}$, $c_1$, and $\eta$ be the constants in Lemma \ref{L6}, and $d$ be the constant in Lemma \ref{L2}. We assume that $C$ is large enough. By Lemma \ref{L2}, there is a sequence of positive integers $i_{1}<\ldots <i_{s}\leq k$ such that $s= \lceil d k/\ln k\rceil$ and the set $\mathcal{L}^{\prime}=\{L_{i_{1}},\ldots, L_{i_{s}}\}$ is admissible. We have
 \[
 s\geq d\,\frac{k}{\ln k}\geq D,
 \] if $C$ is large enough. Also, $s \leq k \leq \eta(\ln x)^{1/4}(\ln\ln x)^{-1/2}$. By Lemma \ref{L6}, there exists a prime $B$ such that
 \[
 0.9 \ln\ln x< B \leq \textup{exp} (c_1 \sqrt{\ln x})
 \]and we have
 \[
  \#\{ x\leq n \leq 2x-2: \#(\{L_{i_{1}}(n),\ldots, L_{i_{s}}(n)\}\cap \mathbb{P}) \geq D^{-1}\ln s\} \geq
 \frac{x}{(\ln x)^{s} \textup{exp}(Ds)}.
 \]Hence,
 \[
 \#\{ x\leq n \leq 2x-2: \#(\{L_{1}(n),\ldots, L_{k}(n)\}\cap \mathbb{P}) \geq D^{-1}\ln s\} \geq
 \frac{x}{(\ln x)^{s} \textup{exp}(Ds)}.
 \]Since $s\leq k$, we obtain
 \[
 \frac{x}{(\ln x)^{s} \textup{exp}(Ds)}\geq \frac{x}{(\ln x)^{k} \textup{exp}(Dk)}.
 \]We have
 \[
 \ln s\geq \ln \Big(d\,\frac{k}{\ln k}\Big) \geq \frac{1}{2}\ln k,
 \]if $C$ is large enough. We may assume that $C\geq 2D$. We obtain
 \begin{equation}\label{T1:RESULT.FOR.T1}
 \#\{ x\leq n \leq 2x-2: \#(\{L_{1}(n),\ldots, L_{k}(n)\}\cap \mathbb{P}) \geq C^{-1}\ln k\} \geq
 \frac{x}{(\ln x)^{k} \textup{exp}(Ck)}.
 \end{equation}Finally, let us denote $B$ by $p_0$ and $\eta$ by $c$. Theorem \ref{T1} is proved.

 \end{proof}

 \begin{proof}[Proof of Theorem \ref{T2}.] Let $C$, $c$, $c_{0}$, and $c_1$ be the constants in Theorem \ref{T1}, and let $c_{2}=c_{0}/2$. We take $p_0$ from Theorem \ref{T1}. We put
 \[
 b^{*}=\max_{1\leq i \leq k} b_{i}
 \]and
 \[
 l_{i}=a_{i}(b^{*} + 1) - b_{i},\qquad F_{i}(n)=a_{i}n+ l_{i}
 \]for any $1 \leq i \leq k$. We see that $\mathcal{F}=\{F_{1},\ldots, F_{k}\}$ is a set of distinct linear functions $F_{i}(n)=a_{i}n+l_{i},$ $i=1,\ldots, k,$ with coefficients in the positive integers, and $(a_{i}, l_{i})=1$ for any $1\leq i \leq k$. We have
 \begin{gather*}
 a_{i}\leq \textup{exp}((c_{0}/2)\sqrt{\ln x})\leq \textup{exp}(c_{0}\sqrt{\ln x}),\qquad (a_i,p_0)=1,\\
 l_{i}\leq \textup{exp}((c_{0}/2)\sqrt{\ln x})(x+1)\leq x\,\textup{exp}(c_{0}\sqrt{\ln x})
 \end{gather*} for all $1 \leq i \leq k$  (if $C$ is large enough). By Theorem \ref{T1}, we have (see also \eqref{T1:RESULT.FOR.T1})
 \[
 \#\{ x\leq n \leq 2x-2: \#(\{F_{1}(n),\ldots, F_{k}(n)\}\cap \mathbb{P}) \geq C^{-1}\ln k\} \geq
 \frac{x}{(\ln x)^{k} \textup{exp}(Ck)}.
 \]For any $x\leq n\leq 2x-2$, we have
 \[F_{i}(n)=a_{i}(n+ b^{*} + 1) - b_{i}=L_{i}(n+ b^{*} + 1)
 \] and $x \leq n+ b^{*} + 1\leq 3x-1$. We obtain
 \[
 \#\{ x\leq n \leq 3x-1: \#(\{L_{1}(n),\ldots, L_{k}(n)\}\cap \mathbb{P}) \geq C^{-1}\ln k\} \geq
 \frac{x}{(\ln x)^{k} \textup{exp}(Ck)}.
 \]Theorem \ref{T2} is proved.

 \end{proof}

\section{Acknowledgements}

The author is grateful to Alexander Kalmynin and to the anonymous referee for useful comments.

This research was supported by Russian Science Foundation, grant
20-11-20203, https://rscf.ru/en/project/20-11-20203/.

\end{document}